\documentclass[12pt]{article}
\usepackage{amsmath,amsthm,amssymb,latexsym,color}
\usepackage{bbm}

\usepackage{graphicx}
\usepackage{tikz}
\usetikzlibrary{decorations.pathreplacing,shapes.misc}

\usepackage{pgfplots}

\usepackage{hyperref}

\date{}

\theoremstyle{plain}

\newtheorem*{theorem*}{Theorem}

\newtheorem{theor}{Theorem}[section]

\newtheorem*{theor*}{Theorem}
\newtheorem*{conj*}{Conjecture}

\newtheorem{thm}[theor]{Theorem}

\newtheorem{prop}[theor]{Proposition}
\newtheorem{defin}[theor]{Definition}

\newtheorem{lemma}[theor]{Lemma}

\theoremstyle{remark}

\def\Bo{{\mathcal B}}

\def\Be{\Bo_{d}(\eps)}
\def\Bae{\Bo_{d}^0(\eps)}
\def\Bep{\wt \Bo_{d}(\eps)}

\def\Ao{{\mathcal A}_d}
\def\Ae{\Ao(\eps)}
\def\Aeg{\Ao (\eps^{1+\ga})}
\def\Aeb{\Ao (\eps^{\bt})}

\def\Lc{{\mathcal L}}
\def\Rc{{\mathcal R}}
\def\Nc{{\mathcal N}}

\def\Cp{{\mathcal C}_+}
\def\Cm{{\mathcal C}_-}
\def\R{{\mathbb R}}
\def\N{{\mathbb N}}
\def\Z{{\mathbb Z}}

\def\Prob{{\mathbb P}}

\def\RR{{\mathbb R}}

\def\Rm{\RR ^m}

\def\r{\right}

\def\wt{\widetilde}

\def\d{{\rm disp}}
\def\kd{ \mbox{\rm $k$-disp}}

\newcommand{\bt}{\beta}
\newcommand{\ga}{\gamma}
\newcommand{\de}{\delta}

\newcommand{\eps}{\varepsilon}

\def\qand{\quad \mbox{ and } \quad}

\title{New bounds on the minimal dispersion}

\author{
A. E. Litvak
\and
G. V. Livshyts
}
\newcommand\address{\noindent\leavevmode

\medskip
\noindent
Alexander E. Litvak\\
Dept.~of Math.~and Stat.~Sciences,\\
University of Alberta, \\
Edmonton, AB, Canada, T6G 2G1.\\
\texttt{\small
e-mail:  aelitvak@gmail.com}\\

\medskip

\noindent
Galyna V. Livshyts,\\
School~of Math., GeorgiaTech,\\
686 Cherry street,\\
Atlanta, GA 30332, USA.\\
\texttt{\small
e-mail:   glivshyts6@math.gatech.edu}\\
}

\begin{document}

\maketitle

\begin{abstract}
 We provide a new construction for a set of boxes approximating
 axis-parallel boxes of fixed volume in $[0, 1]^d$. This improves
 upper bounds for the minimal dispersion of a point set
 in the unit cube and its inverse in both the periodic and non-periodic settings
 in certain regimes. In the case of random choice of points our bounds are sharp
 up to double logarithmic factor.  We also apply our construction to $k$-dispersion.
\end{abstract}

\bigskip

{\small


\section{Introduction}
\label{intro}

The dispersion of a given subset $P$ of the  $d$-dimensional unit cube $[0,1]^d$ is the supremum over
volumes of axis-parallel boxes in the cube that do not intersect $P$, where by an axis-parallel box we
mean a polytope with facets parallel to coordinate hyperplanes. The minimal dispersion is the infimum
of the dispersions of all possible subsets $P\subset [0,1]^d$ of cardinality $n$.
This definition was introduced in \cite{RT}  modifying a  notion from \cite{Hl}.
This notion has many applications in different areas and attracted a significant attention of researchers
in recent years.  We refer to \cite{AHR, DJ, Rud,  MU} and references therein for the
history of estimating the minimal dispersion and relations to other branches of mathematics, to
\cite{BC, HKKR, AEL, KMK, Sos, UV, UV1} for recent developments and best known bounds and to
 \cite{Kr, VT, MU-lat} for the dispersion of certain sets.
 In this note we improve some upper bounds for the minimal dispersion on the cube
 and for its inverse function. We also discuss corresponding bounds on the torus
 and $k$-dispersion (the notion introduced in \cite{HPUV}, which slightly modifies the standard
 definition  by allowing to have at most $k$ points in the intersection of  a given set $P$ and
 an  axis-parallel box). An important feature of our results is that we consider the
 dispersion and its inverse as functions of two variables without fixing one of the parameters.
The improvement of previous results is achieved by a new
construction of an approximating  family of axis-parallel boxes  (periodic or non-periodic)
needed to be checked for a random choice of points.

\subsection{Notation}
\label{notat}

We denote $Q_d:=[0, 1]^d$. We will use the notation $|\cdot|$ for either
cardinality of a finite set or for the  $d$-dimensional volume of a
 measurable subset of $\R^d$ (the precise meaning will be always clear from the context).
The set of all axis-parallel boxes contained in the cube is denoted by
$\Rc _d$, that is
\begin{equation}\label{setrd}
  \Rc _d :=\left\{ \prod _{i=1}^d I_i  \,\,\, |\,\,\,  I_i =[a_i, b_i) \subset [0, 1] \r\}.
\end{equation}
Given a finite set $P\subset Q_d$ its dispersion  is defined as
$$
   \d (P) = \sup \{|B| \, \, \, | \, \, \, B\in \Rc _d, \, B\cap P=\emptyset\} .
$$
The minimal dispersion is defined as the function of two variables $n$ and $d$ as
$$
 \d^* (n, d) = \inf_{P\subset Q_d\atop |P|=n} \d(P) .
$$
Its inverse function is
$$
  N(\eps, d) = \min\{ n\in \N\,  | \,\,   \d^* (n, d)\leq \eps\} .
$$
In this paper it will be more convenient to obtain bounds for  the function $N(\eps, d)$,
then bounds for $\d^* (n, d)$ follow automatically.

Letters $C, C', C_0, C_1, c, c_0, c_1$, etc, always mean absolute positive constants
(that is, numbers independent of any other parameters).

\subsection {Known results.}

We first discuss best bounds in the ``classical" regime when $\eps \to 0$ much faster than $d\to \infty$.
The first upper bound
$$
 N(\eps, d) \leq \frac{2^{d-1}}{n} \, \prod_{i=1}^{d-1}p_i,
$$
where $p_i$ denotes the $i$th prime, was given by  Rote and Tichy \cite{RT} (see also \cite{DJ}).
It was improved by Larcher (see \cite{AHR}) to
$$
    N(\eps, d) \leq \frac{2^{7d+1}}{\eps} .
$$
Very recently it was improved by Bukh and Chao \cite{BC} to
\begin{equation}\label{bcup}
    N(\eps, d) \leq  \frac{C\, d^2 \ln d}{\eps} .
\end{equation}
Since one clearly has $\d^*(n, d) \geq 1/(n+1)$, we have  $N(\eps, d)\geq 1/\eps-1$,
this shows that for a fixed $d$ and $\eps \to 0$, we have $N(\eps, d)\sim C_d/\eps$.
The first lower bound which grows with the dimension was obtained by Aistleitner, Hinrichs,
and Rudolf, who proved that for every $\eps\in (0, 1/4)$,
\begin{equation}\label{ahbd}
    N(\eps, d)\geq \frac{ \log_2 d}{8\eps }
\end{equation}
(this bound is a combination of  Corollary~1 in \cite{AHR} and   Lemma~2 \cite{AHR}, which implies
 $N(\eps, d)\geq \log_2 (d+1)$ whenever $\eps<1/4$). Moreover, Buch and Chao \cite{BC} proved that
 for $\eps \leq (4d)^{-d}$ one has
\begin{equation}\label{b-c-b}
    N(\eps, d)\geq \frac{ d}{e \eps }.
\end{equation}
We would like to note that from results of Dumitrescu and Jiang \cite{DJ, DJ1} (see also \cite{BC}),
it follows that for every $d$ the following limit exists
$$
  \ell_d = \lim _{n\to \infty} n \, \d^* (n, d) .
$$
In particular, from Buch and Chao bounds it follows that $d/e\leq \ell_d\leq C\, d^2 \ln d$.

On the other hand, if we  fix $\eps$ and consider $d\to \infty$, then
the best upper bound is due to Sosnovec \cite{Sos} who proved
for $\eps <1/4$
\begin{equation}\label{uvs}
    N(\eps, d) \leq C'_\eps \log_2 d .
\end{equation}
This bound matches (\ref{ahbd}), showing $N(\eps, d) \sim C_\eps \log_2 d$ for $\eps<1/4$.
The original proof of Sosnovec does not give a good dependence of $C'_\eps$ on $\eps$.
It was improved in \cite{UV} by Ullrich and Vyb\'iral and later in \cite{AEL} by the first named author
to
\begin{equation}\label{ael-sos}
   C'_\eps \leq  C \, \frac{\ln (e/\eps)}{\eps ^2} .
\end{equation}
We also would like to mention that in the same paper Sosnovec showed that for $\eps>1/4$ ,
$
  N(\eps, d) \leq 1+ (\eps -1/4)^{-1},
$
which was improved by MacKay \cite{KMK} to
$$
  N(\eps, d) \leq \frac{\pi}{\sqrt{\eps -1/4}} -3
$$
for $\eps \in (1/4, 1/2)$. For $\eps \geq 1/2$ we have $N(\eps, d)=1$ (it is enough to consider the point $(1/2, ..., 1/2)$).

We finally discuss the case when both $d$ and $1/\eps$ are growing to $\infty$ with a comparable speed.
In \cite{Rud} Rudolf proved
\begin{equation}\label{rudbou}
   N(\eps, d) \leq \frac{8 d}{\eps}\log_2 \left(\frac{33}{\eps}\r).
\end{equation}
This bound with different numerical constants also follows from  much more general
 results in \cite{BEHW}, where the VC dimension
 of $\Rc_d$  was used, and from the fact that this VC dimension
 equals to $2d$).
Rudolf  used a random choice of points uniformly distributed in $Q_d$.
 His bound  is better than the upper bound (\ref{bcup}) in the regime
$$
 \eps \geq \exp(-C \, d\ln d).
$$
Then in \cite{AEL} the first named author proved that for every $d\geq 2$ and  $\eps \leq 1/2$,
\begin{equation}\label{ael-b}
   N(\eps, d) \leq \frac{C\left(\ln d\, \ln(e/\eps) + d\ln \ln (e/\eps)\r)}{\eps} ,
\end{equation}
which is better than the upper bound (\ref{bcup}) for
$$
 \eps \geq \exp(-C \, d^2).
$$

\subsection{New results}\label{ssnewr}

Our main result is

\begin{thm}\label{main}
Let $d\geq 2$ and $\eps \in (0, 1/2]$. Then
$$
     N(\eps, d)\leq 12e \,\, \frac{4 d \ln \ln (8/\eps)+ \ln (1/\eps)}{\eps }.
$$
Moreover,  the random  choice of points
with respect to the uniform distribution on the cube $Q_d$ gives
the result with high probability.
\end{thm}

\medskip

\noindent
{\bf Remarks. 1.  } Let us compare this result to the previously known ones. When $\eps \leq  d^{-d}$, we obtain
$$
  N(\eps, d)\leq \frac{C}{\eps }\, \, \ln \left(\frac{1}{\eps}\r).
$$
This improves the upper bound (\ref{ael-b}) by $\ln d$ factor and is very close to \label{b-c-b}
$\log_2 d/(8\eps)$ given by (\ref{ahbd}). On the other hand, when $\eps \geq  e^{-d}$, we get the same upper
bound as (\ref{ael-b}), namely $12 e d \eps^{-1} \ln \ln (8/\eps)$.

\smallskip

\noindent
{\bf 2. }
We  would like to mention, that Hinrichs, Krieg, Kunsch, and Rudolf \cite{HKKR} investigated
the best bound that one can get using a random choice of points uniformly distributed in the cube.
They showed that one cannot
expect anything better than
\begin{equation}\label{lrb}
  \max \left\{ \frac{c}{\eps} \ln\left(\frac{1}{\eps}\r),\, \frac{d}{2 \eps}\r\}.
\end{equation}
Thus our result is the best possible for this method up to $\ln \ln (e/\eps)$ factor in the first summand.

\smallskip

\noindent
{\bf 3. }
Our proof is also based on a random choice of points and is very similar to
proofs in \cite{Rud, AEL}. In such proofs one tries to produce a finite set of ``test" boxes,
such that if a property (in our case --- each test box contains no random points) holds for every test box,
then the property  holds for all boxes. The simplest way to produce such test boxes is
to create a set of axis-parallel boxes of large enough volume such that each axis-parallel box
of volume $\eps$ contains one test box. Since at the end one uses a union bound it is very important
to control the cardinality of the set of test boxes. Rudolf used the concept of $\delta$-cover
\cite{Rud, MG} for this purpose, while the first named author \cite{AEL} used a more direct construction.
In this paper we suggest another construction which seems right for this problem, see Proposition~\ref{net-gen-box}.
The main idea of this construction comes from a work of the second named author on random matrices \cite{GVL}.
We would also like to mention that, surprisingly, our new construction does not lead to any improvement for
large $\eps$, that is for $\eps \geq 1/d$  --- we may apply our new set of test boxes, but the bound will be the same
as in \cite{AEL}.

\bigskip

 Thus, combining bounds of Theorem~\ref{main} with bounds (\ref{bcup}), (\ref{uvs}), and  (\ref{ael-sos}), the current
 state of the art  can be summarized in
 $$
   N(\eps, d)\leq
 \begin{cases}
     \frac{C\, \ln d }{\eps^2}\, \ln \left(\frac{1}{\eps}\r), &\mbox{ if }\, \eps \geq  \frac{\ln^2 d}{d\ln \ln (2d)},\\
     \frac{C \, d }{\eps}\, \ln \ln \left(\frac{1}{\eps}\r), &\mbox{ if }\, \frac{\ln^2 d}{d\ln \ln (2 d)}\geq \eps \geq d^{-d} ,\\
     \frac{C  }{\eps}\, \ln \left(\frac{1}{\eps}\r),  &\mbox{ if }\, d^{-d}\geq \eps \geq  d^{-d^2},\\
     \frac{C\, d^2 \ln d }{\eps}, &\mbox{ if }\, \eps \leq  d^{-d^2},
 \end{cases}
$$
or in the following picture

\bigskip

$\quad \quad\quad \quad \includegraphics[width=0.8\textwidth]{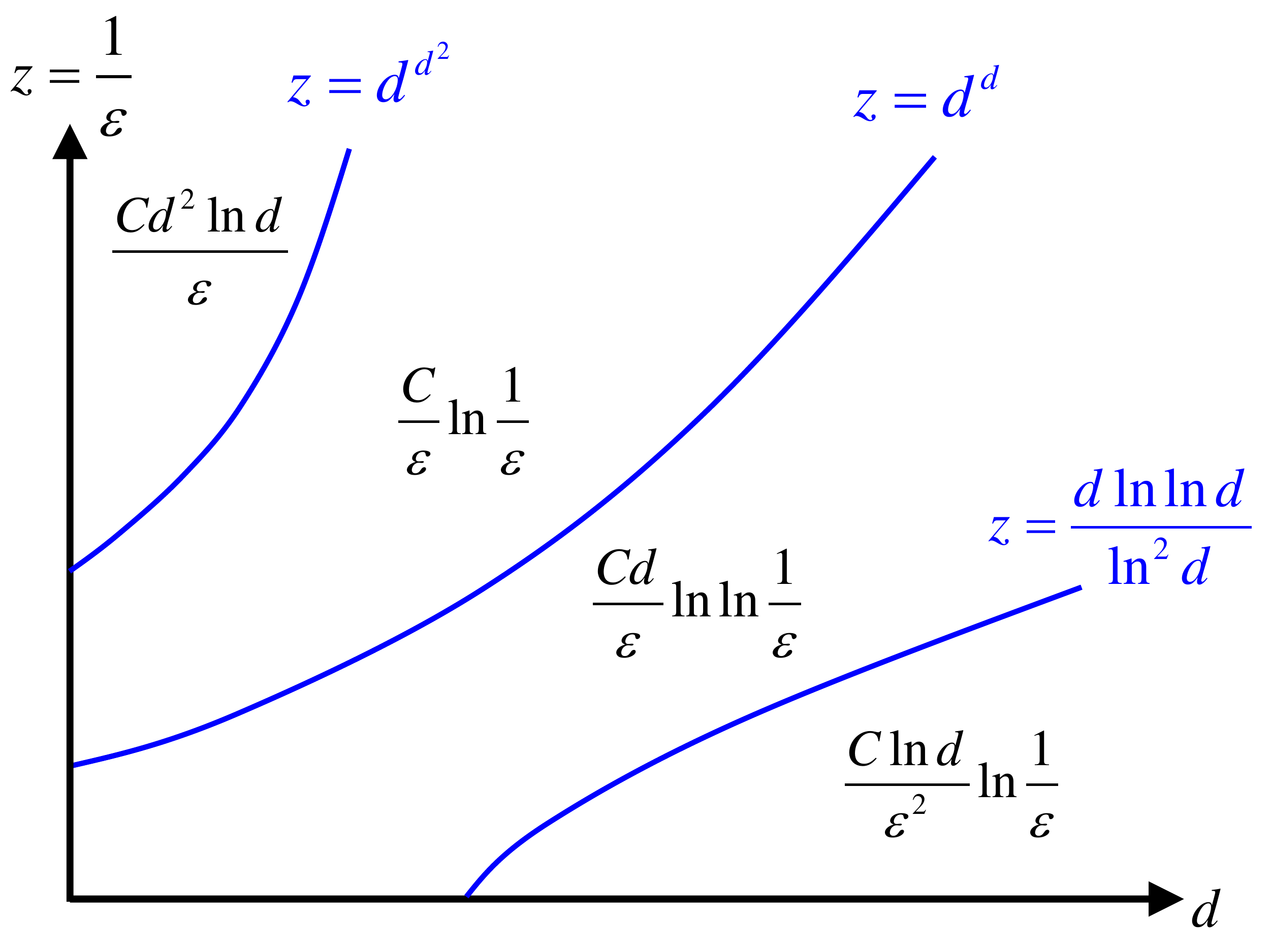}$

\bigskip

Finally, we would like to mention that in terms of the minimal dispersion,
 Theorem~\ref{main} is equivalent to the following theorem.

\begin{thm}\label{disp1}
There exists an absolute constant $C\geq 1$ such that the following holds.
Let $d\geq 2$ and $n\geq 4 d$. Then
$$
  \d^*(n, d)\leq C\,\, \frac{\ln n  +  d\ln \ln (n/d)  }{n } .
$$
Moreover, the random  choice of points
with respect to the uniform distribution on the cube $Q_d$ gives
the result with high probability.
\end{thm}

Recall that in the case $\displaystyle 2\ln d \leq n \leq  \frac{d^2 \ln^ 2\ln d}{\ln^2 d}$, a better bound
was proved in \cite{AEL}, namely Theorem~1.3 there (or combination of (\ref{uvs}) with (\ref{ael-sos}))  gives
$$
  \d^*(n, d)\leq \left(\frac{C \,\ln d}{n }\, \, \ln \left(\frac{n}{\ln d}\r)\r)^{1/2}.
$$

\subsection{Dispersion on the torus}
\label{sstorus}

The   dispersion on the torus can be described in terms of
periodic axis-parallel boxes. We denote such a set by $\wt \Rc _d$, that is
\begin{equation}\label{setperrd}
  \wt \Rc _d :=\left\{ \prod _{i=1}^d I_i(a, b) \,\,\, |\,\,\,  a,  b\in  Q_d \r\},
\end{equation}
 where
$$
   I_i(a, b) :=
   \begin{cases}(a_i, b_i), &\mbox{ whenever }\, 0\leq a_i< b_i\leq 1 ,\\
    [0, 1]\setminus  [b_i, a_i], &\mbox{ whenever }\, 0\leq b_i< a_i\leq 1. \end{cases}
$$
The dispersion of a finite set $P\subset Q_d$ on the torus, the minimal dispersion on the torus,
and its inverse are defined in the same way as above, but using sets from $\wt \Rc _d$,
that is
$$
   \wt \d (T) = \sup \{|B| \, \, \, | \, \, \, B\in \wt \Rc _d, \, B\cap T=\emptyset\},
 \quad \quad  \quad
   \wt \d^* (n, d) = \inf_{|P|=n} \wt \d(P)  ,
$$
and
$$
 \wt N(\eps, d) = \min\{ n\in \N\,  | \,\,  \wt \d^* (n, d)\leq \eps\}  .
$$
 The lower bound
$$
    \wt N(\eps, d)\geq \frac{d}{\eps}
$$
was obtained by  Ullrich \cite{MU}.
We would like to emphasize that contrary to the non-periodic
case, even in the case of large $\eps$, the lower bound is at least $d$.
 The upper bound
\begin{equation}\label{per-ael}
  \wt N(\eps, d)\leq \frac{C\, \ln d \, \left(d + \ln (e/\eps)\r)}{\eps} ,
\end{equation}
was obtained by the first named author \cite{AEL}, who improved Rudolf's bound \cite{Rud}
 $(8 d/\eps) \, \left(\ln d +  \ln (8/\eps)\r)$.
Note that since  the VC dimension of
$\wt \Rc _d$ is not linear in $d$ \cite{GLM}, results of \cite{BEHW} would lead
to  worse bounds.
We improve upper bound (\ref{per-ael}) in the case $\eps \leq 1/d$ by removing the factor
$\ln d$ in front of the second summand.

\begin{thm}\label{th-per}
Let $d\geq 2$ and $\eps \in (0, 1/2]$. Then
$$
  \wt  N(\eps, d)\leq 24e\,  \frac{2 d \ln (2d)+ \ln (e/\eps)}{\eps }.
$$
Moreover,  the random  choice of points
with respect to the uniform distribution on the cube $Q_d$ gives
the result with high probability. Equivalently, there exists an absolute constant
$C>1$ such that  for $d\geq 2$ and $n\geq 1$ one has
$$
  \wt   \d^*(n, d)\leq  C\, \frac{d\ln d + \ln n}{n }.
$$
\end{thm}

\bigskip

The proof is essentially the same as for Theorem~\ref{main},
but some adjustments are required in the construction of approximating sets. This leads to a slightly
worse bound. See the remark preceding Proposition~\ref{net-tor-box} for the details.
 We would also like to note that the
Hinrichs--Krieg--Kunsch--Rudolf's result on best possible lower bound (\ref{lrb}) which may be obtained
by using random points uniformly distributed on the cube holds for the periodic setting as well,
therefore the summand $\ln (e/\eps)$  is unavoidable by this method.

\section{Preliminaries}
\label{intro}

Given a positive integer $m$ we denote $[m]=\{1, 2, ..., m\}$.
Recall that the sets $\Rc _d$ and $\wt \Rc _d$ were defined in (\ref{setrd}) and (\ref{setperrd})
respectively. Given $\eps >0$, we consider sets of (periodic) axis-parallel boxes of volume at least $\eps$,
$$
   \Be :=\Big\{ B\in \Rc _d   \,\,\, |\,\,\,  |B|\geq \eps \Big\}
   \qand  \Bep :=\Big\{ B\in \wt \Rc _d   \,\,\, |\,\,\,  |B|\geq \eps \Big\}.
$$
We also consider {\it anchored} axis-parallel boxes (that is, containing the origin as a vertex), defined as
\begin{equation}\label{anc-def}
   \Bae :=\Big\{ B\in \Rc _d^0   \,\,\, |\,\,\,  |B|\geq \eps \Big\}, \, \mbox{ where } \,
  \Rc _d^0 :=\left\{ \prod _{i=1}^d I_i  \,\,\, |\,\,\,  I_i =[0, b_i) \subset [0, 1] \r\}.
\end{equation}

\begin{defin}[$\delta$-approximation for $\Be$] \label{delta-net}
 Given $0<\delta\leq \eps \leq 1$ we say that $\Nc \subset \Rc _d$
is a $\delta$-approximation for $\Be$ if for every $B\in \Be$
there exists $B_0\in \Nc$ such that $B_0\subset B$ and
$$
 |B_0|\geq \delta.
$$
We define a $\delta$-approximation for $\Bae$ and $\Bep$ in a similar way.
\end{defin}

\medskip

\noindent
{\bf Remark.  }
This definition is a slight modification of the notions of $\delta$-net and $\delta$-dinet from
\cite{AEL}. An essentially the same notion was recently considered in a similar context by M.~Gnewuch
\cite{MGn}.

\medskip

A variant of the following lemma using random points and the union bound
was proved in  \cite{Rud} (see Theorem~1 there). We will use the following
formulation taken from \cite{AEL} (see Lemma~2.3 and Remark~2.4
there). The proof in \cite{AEL} was provided for $\delta$-nets, but it is easy
to check that the same proof works for $\delta$-approximations.

\begin{lemma}\label{unb}
Let $d\geq 1$ and  $\eps, \delta  \in (0,1)$. Let $\Nc$ be  a $\delta$-approximation for $\Be$
and
let   $\wt \Nc$ be a $\delta$-approximation for $\Bep$. Assume both $|\Nc|\geq 3$ and $|\wt \Nc|\geq 3$.
 Then
$$
   N(\eps, d) \leq  \frac{3\ln |\Nc|}{\delta} \quad  \qand\quad   \wt N(\eps, d)\leq \frac{3\ln |\wt \Nc|}{\delta } .
$$
Moreover, the random choice of independent points (with respect to the uniform distribution on $Q_d$)
 gives the result with probability at least  $1-1/|\Nc|$.
\end{lemma}

We finally discuss covering numbers.
Let  $K$ and  $L$ be subsets of a linear space $X$. The covering number
$N (K, L)$ is defined as the smallest integer $N$ such that there are
$x_1$, ..., $x_N$ in $X$ satisfying
\begin{equation}
\label{covering}
 K\subset \bigcup _{i=1}^N (x_i +L).
\end{equation}

For a convex body $K \subset \RR^m$ and $\gamma \in (0,1)$, we will
need an upper bound for the covering number $N (K, -\ga K)$. We could
use a standard volume argument, which would be sufficient for our results,
but we prefer to use a more sophisticated estimate by Rogers-Zong
\cite{RZ}, which leads to slightly better constants.

Let $m\geq 1$,   we set  $\theta _m=\sup \theta (K)$, where
the supremum is taken over all convex bodies $K\subset \RR ^m$
and $\theta (K)$ is the covering density
of $K$ (see \cite{Rog} for the definition and more details). It is
known (see \cite{Ro}, \cite{Rog}) that $\theta _1=1$, $\theta
_2\leq 1.5$, and, by a result of Rogers,
$$
\theta _m \leq \inf_{0<x<1/m} (1+x)^m (1- m \ln x) <
m (\ln m +  \ln \ln m + 5)
$$
for $m \ge 3$. We will use following lemma from \cite{RZ}.

\begin{lemma} \label{rogcov}
Let $m>1$ and $K$ and $L$ be two convex bodies in $\Rm$.
Then
$$
  N  (K, L) \leq  \theta _m \frac{|K-L|}{|L|} ,
$$
in particular, for every $\ga >0$.
$$
  N (K, -\ga K) \leq  7 m\ln m  \left(\frac{1+\ga}{\ga}\r)^m .
$$
\end{lemma}

\section{Cardinality of approximating sets}
\label{card}

We start with anchored boxes. The following lemma in a more general setting was proved in
\cite{GVL} (see Lemma~3.10 there). We provide a direct proof in our setting. Recall that
$\Bae$ was defined by (\ref{anc-def}).

\begin{prop}\label{net-anc}
Let $d\geq 2$ be an integer, $\eps \in (0,1)$, and $\ga>0$. Let $\delta = \eps^{1+\ga}$.
Then the size of an optimal $(\eps^{1+\gamma})$-approximation of $\Bae$
equals to
$$
   N(S_{d-1}, - \ga S_{d-1})\leq  7 d\ln d  \left(\frac{1+\ga}{\ga}\r)^{d-1},
$$
where $S_{d-1}$ is a regular $(d-1)$-dimensional simplex.
\end{prop}

\begin{proof}
We first identify each box in $\Rc_d^0$ with its right upper corner, that is,
each box $B=\prod _{i=1}^{d} [0, b_i)$ we identify with $b=\{b_i\}_{i=1}^{d}$.
Since each box $B\in \Bae$ contains an anchored box of volume precisely equal to
$\eps$ we may restrict ourself to considering only boxes of volume $\eps$.

For $\beta \geq 1$ consider the sets
$$
   \Aeb=\left\{b=\{b_i\}_{i=1}^{d}\in Q_d \,\, | \, \,  \prod _{i=1}^{d}  b_i = \eps^\beta \r\}
$$
(we use them with $\beta =1$ and $\beta =1+\ga$). It is enough to prove that
there exists a set  $\Nc_0\subset \Aeg$ (of an appropriate cardinality)
such that for every $b=\{b_i\}_{i=1}^{d}\in \Ae$ there exists
$a=\{a_i\}_{i=1}^{d}\in \Nc_0$ satisfying $a_i\leq b_i$ for every $i\leq d$.

Consider the function $f_\eps \, : \, (0,1] \to [0, \infty)$ defined by
$$
    f_\eps (t) = \frac{\ln (1/t)}{\ln (1/\eps)}.
$$
 Note that if $b_i\geq 0$, $i\leq d$, are such that $\prod_i b_i = \eps^\beta$, then
 $$
   \sum_{i=1}^d f_\eps (b_i) = \sum_{i=1}^d  \frac{\ln (1/b_i)}{\ln (1/\eps)}
   =  \frac{1}{\ln (1/\eps)}\,\,  \ln \prod_{i=1}^d  (1/b_i) =\beta.
 $$
 Let $F_\eps \, : \, (0, 1]^d \to (0, \infty]^d$  be defined by  $F_\eps(\{x_i\}_{i=1}^{d}) = \{f_\eps (x_i)\}_{i=1}^{d}$.
Denote
\begin{align*}
   \Cp&:=\{ x=\{x_i\}_{i=1}^{d}\in \R^d \,\, |\, \,\forall i\leq d : \, x_i \geq 0\},   \\
   \Cm&:=\{ x=\{x_i\}_{i=1}^{d}\in \R^d \, \, |\,\, \forall i\leq d : \, x_i \leq 0\}, \quad \mbox{ and } \\
   H&:=\{ x=\{x_i\}_{i=1}^{d}\in \R^d \,\, |\,\, \sum_{i=1}^d x_i =1 \}.
\end{align*}
Note that for each fixed $\beta \geq 1$ the function $F_\eps$ is a bijection between
$\Aeb$ and $\bt H \cap \Cp$. Thus it is enough to check that
there exists a set  $\Nc_1 \subset  (1+\ga) H \cap \Cp$ such that for every $x=\{x_i\}_{i=1}^{d}\in  H \cap \Cp$ there exists
$y=\{y_i\}_{i=1}^{d}\in \Nc_1$ satisfying $y_i\geq x_i$ for every $i\leq d$ (note here that $f_\eps$ is a decreasing function).

Identify $H$ with $(d-1)$-dimensional Euclidean space $X$ centered at $e:=(1/n, ..., 1/n)$. Let $S_{d-1} = H \cap \Cp$ (the regular
simplex with vertices at the standard basis vectors of $\R^d$).
Given $y\in   (1+\ga) H \cap \Cp$ consider the set $S(y) := (y+\Cm) \cap H$. Then $y$ serves as an approximation for all points
in $S(y)\cap S_{d-1}$ in the above sense, that is, for every point $x\in S(y)$ we have $y_i\geq x_i$ for every $i\leq d$. In other words, we need
to estimate the cardinality of a (minimal) set of $y$'s such that the sets $S(y)$ cover $S_{d-1}$.  Noticing that $S(y)$ is a shift
of $-\gamma S_{d-1}$ (where the multiplication of $S_{d-1}$ by $-\gamma$ is taken in $X$ with respect to the origin $e$), this means
that we have to estimate the covering number $N(S_{d-1}, - \ga S_{d-1})$. Applying  Lemma~\ref{rogcov} we complete the proof.
\end{proof}

\medskip

Next we obtain a bound for cardinality of  a $\de$-approximation for $\Be$.

\begin{prop}\label{net-gen-box}
Let $d\geq 2$ be an integer, $\eps \in (0,1)$, and $\ga>0$. Let $\delta = \eps^{1+\ga}/4$.
There exists a $\de$-approximation $\Nc$  for $\Be$ of cardinality at most
$$
     7 d\ln d \, \frac{ (1+1/\ga)^{d }(\ln (e/\eps^{1+\ga}))^d}{\eps^{1+\ga} }  .
$$
\end{prop}

\begin{proof}
Let $\delta_0 = \eps^{1+\ga}$ and $\Nc_0$  be $\de_0$-approximation  for $\Bae$ of cardinality at most
$7 d\ln d  \left(1+1/\ga\r)^d$ constructed in Proposition~\ref{net-anc}.
In order to construct $\de$-approximation for $\Be$ we consider shifts of
multiples of boxes from  $\Nc_0$.
For each $B= \prod _{i=1}^{d} [0, b_i)\in  \Nc_0$ we consider the following set of points, that will be used for
shifts,
$$
   \Lc _{B} := \left\{ \{k_i b_i /d\}_{i=1}^{d} \, \, \, | \,  \, \, \forall i\leq d : \, k_i\in \Z, \, 1\leq k_i \leq 1-d+d/b_i\r\}.
$$
Denoting $c_d=1-1/d$ and  using that  $\prod _{i=1}^{d}  b_i = \de_0$
and the inequality between arithmetic and geometric means twice,
we observe
\begin{align*}
  \left|\Lc _B\r| &\leq \prod _{i=1}^{d}  \left(1-d +\frac{d}{b_i}\r) = \prod _{i=1}^{d} \frac{d}{b_i} \left(1-c_d b_i\r) \leq
  \frac{d^d}{\de_0 }\,  \left(1-\frac{c_d}{d}\sum _{i=1}^{d} b_i\r)^d\\
  &\leq \frac{d^d}{\de_0 }\,  \left(1- c_d \left(\prod _{i=1}^{d} b_i\r)^{1/d}\r)^d =
  \frac{d^d}{\de_0 }\,  \left(1- c_d \de_0^{1/d}\r)^d.
\end{align*}
Since
$$
    c_d \de_0^{1/d}\geq \left(1-\frac{1}{d}\r)\left(1-\frac{\ln (1/\de_0)}{d}\r) \geq 1- \frac{\ln (e/\de_0)}{d},
$$
we obtain $$\left|\Lc _B\r|\leq \frac{(\ln (e/\de_0))^d}{\de_0}.$$

Next consider a box
$$K= \prod _{i=1}^{d} [x_i, y_i)\in  \Be.$$
Denote $x=\{x_i\}_{i=1}^{d}$, $y=\{y_i\}_{i=1}^{d}$, and $a=\{a_i\}_{i=1}^{d}=y-x.$
Then $K=x+ A$, where
$$A= \prod _{i=1}^{d} [0, a_i)\in  \Bae.$$
 Let
$$
  B= \prod _{i=1}^{d} [0, b_i)\in  \Nc_0
$$
be a box which $\delta_0$-approximates $A$.  Then, since $x+B \subset x+A=K\subset  Q_d$,
we  have $0\leq x_i\leq 1-b_i$ for all $i\leq d$. Therefore,
 for every $i\leq d$ there exists a positive integer $k_i(x)$ such that
 $$
      \frac{(k_i(x)-1)b_i}{d} \leq x_i< \frac{k_i(x)  b_i}{d}   \qand \frac{k_i(x)  b_i}{d} \leq 1-b_i+ \frac{b_i}{d}.
 $$
Take $z_i = k_i(x)  b_i/d$ and $z=  \{z_i\}_{i=1}^{d}$. Then $z\in \Lc _{B}$ and
$$
   K\supset \prod _{i=1}^{d} [x_i, x_i+b_i) \supset \prod _{i=1}^{d} [z_i, z_i+c_d b_i) = z+
    \prod _{i=1}^{d} [0, c_d b_i) = z+c_dB.
$$
This implies that
$$
       \Nc : = \bigcup_{B\in \Nc_0}\, \bigcup_{z\in \Lc _{B}}\, (z + c_d B)
$$
is $(c_d^d \, \de_0)$-approximation for $\Be$ of cardinality
$$
  |\Nc| \leq |\Nc_0 |\,  \frac{(\ln (e/\de_0))^d}{\de_0} \leq
  7 d\ln d \, \frac{ (1+1/\ga)^{d }(\ln (e/\eps^{1+\ga}))^d}{\eps^{1+\ga} } .
$$
Since $c_d^d\geq 1/4$ for $d\geq 2$, this implies the desired result.
\end{proof}

\bigskip

\noindent
{\bf Remark. } Note that dealing with periodic boxes and having a periodic box $x+\prod _{i=1}^{d} [0, b_i)$
we cannot conclude that $x_i +b_i \leq 1$, therefore, in the proof above, we have to consider  all possible
$x_i\leq 1$. Thus, for each box $B\in \Bae$ we will have to adjust the definition of $\Lc _{B}$ to
$$
  \Lc _{B}:= \left\{ y =\{k_i b_i /d\}_{i=1}^{d} \, \, |  \, \, \forall i\leq d : \, k_i\in \Z, \, 1\leq k_i \leq 1+d/b_i\r\}.
$$
This will change the upper bound of cardinality of $\Lc _{B}$ to
\begin{align*}
  |\Lc _{B}|= \prod _{i=1}^{d}  \left(1 +\frac{d}{b_i}\r) &\leq  \prod _{i=1}^{d} \frac{2d}{b_i} = \frac{(2d)^d}{\de_0}.
\end{align*}
The rest of the proof will be same with minor adjustments to the periodic intervals. This will lead to the following proposition.

\begin{prop}\label{net-tor-box}
Let $d\geq 2$ be an integer, $\eps \in (0,1)$, and $\ga>0$. Let $\delta = \eps^{1+\ga}/4$.
There exists a $\de$-approximation $\Nc$  for $\Bep$ of cardinality at most
$$
     7 d\ln d \, \frac{ (1+1/\ga)^{d }(2d)^d}{\eps^{1+\ga} }  .
$$
\end{prop}

Propositions \ref{net-gen-box} and \ref{net-tor-box} together with Lemma~\ref{unb} immediately imply
the main results. We provide proofs for completeness.

\begin{proof}[Proof of Theorems~\ref{main} and \ref{th-per}.]
We choose $\gamma =1/\ln (1/\eps)$, so that $\eps ^{1+\ga}= \eps/e$. Let $\delta = \eps^{1+\ga}/4=\eps/(4e)$.
Let $\Nc$ and $\Nc'$ be $\delta$-approximations constructed in Propositions~\ref{net-gen-box} and \ref{net-tor-box} with
cardinalities
$$
  |\Nc|  \leq
  7 d\ln d \, \frac{ (1+1/\ga)^{d }  \, (\ln (e/\eps^{1+1/\ga}))^d}{\eps^{1+\ga}}
  \leq  7e d\ln d \,  \frac{(\ln (e/\eps))^{d }  \, (\ln (e^2/\eps))^d}{ \eps}
$$
and
$$
  |\Nc'|  \leq
  7 d\ln d \, \frac{ (1+1/\ga)^{d }(2d)^d}{\eps^{1+\ga} }
  \leq  7e d\ln d \,  \frac{(\ln (e/\eps))^{d }  \, (2d)^d}{ \eps}
$$
Thus
$$
  \ln |\Nc| \leq 2d \ln \ln (e^2/\eps)+ \ln (1/\eps) + \ln (7e  d\ln d)\leq 4d \ln \ln (8/\eps)+ \ln (1/\eps)
$$
and
\begin{align*}
  \ln |\Nc'| &\leq d \ln \ln (e/\eps)+ \ln (1/\eps) + d\ln (2d) + \ln (7e  d\ln d)
  \\&  \leq 2d\ln \ln (e/\eps) + 2d \ln (2d) +\ln (1/\eps)
  \\&  \leq 4d \ln (2d)+2 \ln (e/\eps) .
\end{align*}
Lemma~\ref{unb} implies the result.
\end{proof}

\section{k-dispersion}
\label{k-disp}
Following \cite{HPUV}, given $k\geq 0$ and a finite set $P\subset Q_d$ we define its $k$-dispersion as

$$
   \kd (P) = \sup \{|B| \, \, \, | \, \, \, B\in \Rc _d, \, |B\cap P|\leq k\} .
$$
In this way the standard dispersion is $0$-dispersion. A similar notion in the context
of star discrepancy of a given set and anchored boxes was considered in \cite{ET, GSW}.
Then the minimal $k$-dispersion is defined as the function of two variables $n$ and $d$ as
$$
 \kd^* (n, d) = \inf_{P\subset Q_d\atop |P|=n} \kd(P) .
$$
Clearly, if $k\geq n$ then $\kd^* (n, d) =1$, therefore we consider $k\leq n$ only.
Moreover, by partitioning  $Q_d$ in two axis-parallel
boxes of volume $1/2$, we immediately get that,
\begin{equation}\label{largek}
    1/2\leq \kd^* (n, d) \leq 1 \quad \mbox{ for }  \quad n/2\leq k \leq n.
\end{equation}

As above, we will work with its inverse,
$$
  N_k(\eps, d) = \min\{ n\geq k\,  | \,\,   \kd^* (n, d)\leq \eps\} .
$$
In \cite{HPUV} the following bound was proved
$$
 \frac{1}{8}\, \min\left\{1, \frac{k +\log_2 d}{n}\r\}\leq
 \kd^*(n, d) \leq C \max\left\{ \ln n \sqrt\frac{\ln d}{n}, \, \frac{k \ln (n/k)}{n}\r\},
$$
or, equivalently,
$$
 c\,\,  \frac{k +\log_2 d}{\eps}  \leq  N_k (\eps, d)\leq C \max\left\{  \frac{\ln^2(e/\eps) }{\eps^2}\,\, \ln d , \, \frac{k \ln (e/\eps)}{\eps}\r\}.
$$
Note that in the cases $k\leq \ln d$ or $k>\ln d$ and  $\eps \leq \frac{\ln d}{k}$ the upper bound behaves as $((\ln (e/\eps))/\eps)^2 \ln d$
which cannot be sharp as $\eps \to 0$. We improve the upper bound in the next theorem.

\begin{thm}\label{main-torus}
Let $d\geq 2$, $k\geq 0$,  and $\eps \in (0, 1/2]$. Then
$$
     N_k (\eps, d)\leq 80e \, \frac{d \ln \ln (8/\eps)+ k \ln (e/\eps)}{\eps }.
$$
Moreover, the random choice of independent points (with respect to the uniform distribution on $Q_d$)
 gives the result with probability tending to 1 as either $d\to \infty$ or $\eps\to 0$.
Equivalently, there exists an absolute constant $C>0$ such that
 for $n\geq 4d$ and $k\leq n/2$, one has
$$
  \kd^*(n, d)\leq C\,\, \frac{k \ln (n/k)  +  d\ln \ln (n/d)  }{n } .
$$
\end{thm}

Note that for $k=0$ this is Theorem~\ref{main} and that in view of (\ref{largek}), we don't consider
$k\geq n/2$ in the ``moreover" part of the theorem. The proof of Theorem~\ref{main-torus} for $k\geq 1$
repeats the proof of Theorem~\ref{main}, we just need to slightly adjust Lemma~\ref{unb} in the  following  way.

\begin{lemma}\label{k-unb}
Let $d\geq 1$, $k\geq 1$, and  $\eps, \delta  \in (0,1)$. Let $\Nc$ be  a $\delta$-approximation for $\Be$ such that $|\Nc|\geq 3$.
 Then
$$
   N_k(\eps, d) \leq  \frac{5}{\delta} \left( \ln |\Nc| + k\ln (e/\delta)\r)   .
$$
Moreover, the random choice of independent points (with respect to the uniform distribution on $Q_d$)
 gives the result with probability at least  $1-1/|\Nc|$.
\end{lemma}

\begin{proof}
Let $\Nc$ be a $\delta$-approximation for $\Be$. Consider $N$ independent random points
$X_1$, ..., $X_N$ uniformly chosen from $Q_d$. By the definition of a $\delta$-approximation, it is enough
to show that with the required probability, there exists a realization of $X_i$'s with the following property:
every $B\in \Nc$ with $|B|\geq \delta $ contains at least $k+1$ points. Fix a box $B\in \Nc$.
Let $\mathcal E$ be the event that $B$ contains at most $k$ points out of $X_i$'s.
Then there exists $A\subset [N]$ with cardinality $|A|=N-k$ such that for every $j\in A$, $X_j\not\in B$.
Thus
\begin{align*}
  \Prob \left(\mathcal E\r) &\leq \Prob \left( \left\{\exists A\subset [N] \,\,\, | \,\, \, |A|= N-k, \, \, \forall j\in A :\,  X_j\notin B \r\}\r)
  \\&
  \leq \sum_{ A\subset [N] \atop |A|= N-k } \Prob \left(  \forall j\in A :\,  X_j\notin B \r) \leq {N \choose k} \, \left(1 - \delta\r)^{N-k}
  \\&
  < \left(\frac{eN}{k}\r)^k \, \exp(-(N-k)\delta).
\end{align*}
Therefore, by the union bound,
$$
  \Prob \left( \left\{\exists B\in \Nc :   \, \, \mbox{$B$ contains at most $k$ points} \r\}\r)
  <
   |\Nc|\left(\frac{eN}{k}\r)^k \, \exp(-(N-k)\delta).
$$
Thus, as far as $|\Nc|\left(\frac{eN}{k}\r)^k \, \exp(-(N-k)\delta)\leq 1/|\Nc|$, $X_j$'s satisfy the desired property with required probability. This inequality is equivalent to
\begin{equation}\label{lastin}
  2 \ln |\Nc| + k \ln \frac{eN}{k} \leq \delta (N-k).
\end{equation}
It remains to show that
$$
   N= \left\lfloor \frac{5\ln |\Nc|}{\delta} + \frac{5k\ln (e/\delta)}{\delta}   \r\rfloor
$$
satisfies (\ref{lastin}). First note that such a choice of $N$ satisfies $N\geq 5k$, hence
\begin{equation}\label{lastin-1}
  \delta (N-k) \geq \frac{4 \delta N}{5}\geq 4 \ln |\Nc|.
\end{equation}
We have also
$N/k\geq 5 \delta ^{-1} \,  \ln (e/\delta)$.  Using that $f(x)= x/(\ln (ex))$ is increasing on $(1, \infty)$),
the latter inequality implies that
$N/k\geq 2.5 \delta ^{-1} \,  \ln (eN/k)$.
This leads to
  \begin{equation}\label{lastin-2}
  \delta (N-k) \geq \frac{4 \delta N}{5}\geq k \ln \frac{eN}{k}.
\end{equation}
Since (\ref{lastin-1}) and (\ref{lastin-2}) yield (\ref{lastin}),
this completes the proof.
\end{proof}

\subsection*{Acknowledgments}
We would like to thank Michael Gnewuch for his remarks on the first draft of this paper as well as
for showing us references \cite{GSW, ET}. We also thankful for referees for careful reading.

\address


\nocite{*}






\begin{thebibliography}{99}

\bibitem{AHR}
C.~Aistleitner, A.~Hinrichs, D.~Rudolf, {\it On the size of the largest empty box amidst a point
set}, Discrete Appl. Math. {\bf 230} (2017), 146--150.

\bibitem{BEHW}
A.~Blumer, A.~Ehrenfeucht, D.~Haussler, M.~Warmuth,
{\it Learnability and the Vapnik--Chervonenkis dimension},
J. Assoc. Comput. Mach. {\bf 36} (1989), 929--965.

\bibitem{BC}
B.~Bukh,  T.~Chao,
{\it Empty axis-parallel boxes},
preprint, arXiv:2009.05820.



\bibitem{DJ}
A.~Dumitrescu, M.~Jiang, {\it On the largest empty axis-parallel box amidst $n$ points},
Algorithmica {\bf 66} (2013), 225--248.

\bibitem{DJ1}
A.~Dumitrescu, M.~Jiang,
 Computational geometry column 60. ACM SIGACT  News, {\bf 45} (2014) 76--82.

\bibitem{GLM} P. Gillibert, T. Lachmann, C. M{\"u}llner
{\it The VC-Dimension of Axis-Parallel Boxes on the Torus},
https://arxiv.org/pdf/2004.13861.pdf


\bibitem{MG}
M.~Gnewuch, {\it Bracketing numbers for axis-parallel boxes and applications to geometric discrepancy},
J. Complexity {\bf 24} (2008), 154--172.




\bibitem{MGn}
M.~Gnewuch, private communications.

\bibitem{GSW}
M.~Gnewuch, A.~Srivastav, C.~Winzen,
{\it Finding optimal volume subintervals with $k$-points and calculating the star discrepancy are NP-hard problems},
 J.~Complexity {\bf 25} (2009),  115--127.


\bibitem{HKKR} A.~Hinrichs, D.~Krieg, R.J.~Kunsch, D.~Rudolf,
{\it Expected dispersion of uniformly distributed points,}
J. Complexity, {\bf  61} (2020), 101483




\bibitem{HPUV} A.~Hinrichs, J.~Prochno, M.~Ullrich, J.~Vyb\'iral, {\it
The minimal k-dispersion of point sets in high-dimensions},
J. Complexity, {\bf 51} (2019), 68--78.






\bibitem{Hl}
E.~Hlawka, {\it Absch\"atzung von trigonometrischen Summen mittels diophantischer Approximationen,}
\"Osterreich. Akad. Wiss. Math.-Naturwiss. Kl. S.-B. II, {\bf 185} (1976), 43--50.


\bibitem{Kr}
D.~Krieg, {\it On the dispersion of sparse grids}, J. Complexity {\bf 45} (2018), 115--119.

\bibitem{AEL}
A.E. Litvak, {\it A remark on the minimal dispersion} Commun. Contemp. Math.,
{\bf 23}  (2021),  2050060.

\bibitem{GVL}
G.V. Livshyts, {\it The smallest singular value of heavy-tailed not necessarily i.i.d.
random matrices via random rounding}, Journal d'Analyse Mathematique, {\bf 145} (2021), 257--306.


\bibitem{KMK} K. MacKay, {\it Minimal dispersion of large volume boxes in the cube},
 J. Complexity, to appear.


\bibitem{Ro} C.~A.~Rogers, {\it A note on coverings}, Mathematica,
  {\bf 4} (1957), 1--6.
%
\bibitem{Rog} C.~A.~Rogers, { Packing and Covering}, Cambridge
  Tracts in Mathematics and Mathematical Physics., No. 54, Cambridge:
  University Press 1964.

\bibitem{RZ} C.A.~Rogers, C.~Zong, {\it Covering convex bodies by
translates of convex bodies}, Mathematica, {\bf 44} (1997), 215--218.

\bibitem{RT}
 G.~Rote, R.F.~Tichy, {\it Quasi-Monte Carlo methods and the dispersion of point sequences},
 Math. Comput. Modelling {\bf 23} (1996), 9--23.

\bibitem{Rud} D.~Rudolf, {\it An upper bound of the minimal dispersion via delta covers,}
Contemporary Computational Mathematics - A Celebration of the 80th Birthday of Ian Sloan,
Springer-Verlag, (2018), 1099--1108.



\bibitem{Sos}
J.~Sosnovec, {\it A note on the minimal dispersion of point sets in the unit cube},
European J. of Comb., {\bf  69} (2018), 255--259.

\bibitem{VT}
V.N.~Temlyakov, {\it Dispersion of the Fibonacci and the Frolov point sets},
preprint, 2017,  arXiv:1709.08158.

\bibitem{ET}
E. Thi\'emard,
{\it Optimal volume subintervals with $k$ points and star discrepancy via integer programming},
 Math. Methods Oper. Res. {\bf 54} (2001),  21--45.


\bibitem{MU}
M.~Ullrich, {\it  A lower bound for the dispersion on the torus}, Mathematics and
Computers in Simulation {\bf  143} (2018), 186--190.

\bibitem{MU-lat}
M.~Ullrich,
{\it  A note on the dispersion of admissible lattices},
Discrete Appl. Math., {\bf 257} (2019), 385--387.

\bibitem{UV}
M.~Ullrich, J.~Vyb\'iral, {\it An upper bound on the minimal dispersion,}
Journal of Complexity {\bf 45} (2018), 120--126.

\bibitem{UV1}
M.~Ullrich, J.~Vyb\'iral, {\it
Deterministic constructions of high-dimensional sets with small dispersion}, Preprint, 2019,
arXiv:1901.06702



\end{thebibliography}
\end{document}